\documentclass[12pt]{amsart}

\usepackage{amsmath,amssymb,amsthm,enumerate}

\newtheorem{theorem}{Theorem}[section]
\newtheorem{corollary}[theorem]{Corollary}

\newtheorem{proposition}[theorem]{Proposition}

\theoremstyle{definition}

\theoremstyle{remark}
\newtheorem{remark}[theorem]{Remark}

\newcommand{\bP}{\mathbb P}
\newcommand{\bC}{\mathbb C}
\newcommand{\bR}{\mathbb R}

\newcommand{\bB}{\mathbb B}

\def\beq{\begin{equation}}
\def\eeq{\end{equation}}

\newcommand{\norm}{|\!|}

\DeclareMathOperator{\im}{Im}
\DeclareMathOperator{\re}{Re}

\begin{document}
\title[The Log term in the Bergman and Szeg\H o kernels]{The Log term in the Bergman and Szeg\H o kernels in strictly pseudoconvex domains in $\bC^2$}
\author{Peter Ebenfelt}
\address{Department of Mathematics, University of California at San Diego, La Jolla, CA 92093-0112}
\email{pebenfel@math.ucsd.edu}

\abstract In this paper, we consider bounded strictly pseudoconvex domains $D\subset \bC^2$ with smooth boundary $M=M^3:=\partial D$. If we consider the asymptotic expansion of the Bergman kernel on the diagonal
$$
K_B\sim \frac{\phi_B}{\rho^{n+1}}+\psi_B\log\rho,
$$
where $\rho>0$ is a Fefferman defining equation for $D$, then it is well known that the trace of the log term $b\psi_B:=(\psi_B)|_M$ on $M$ does not determine the CR geometry of $M$ locally; e.g., the vanishing of $b\psi_B$ on an open subset of $M$ does not imply that $M$ is locally spherical there. Nevertheless, the main result in this paper is that if $D\subset \bC^2$ is assumed to have transverse symmetry, then the global vanishing of $b\psi_B$ on $M$ implies that $M$ is locally spherical. A similar result is proved for the Szeg\H o kernel. 
\endabstract

\thanks{The author was supported in part by the NSF grant DMS-1600701.}


\maketitle

\section{Introduction}

Let $D\subset\bC^n$ be a bounded strictly pseudoconvex domain with smooth boundary $\partial D$, and assume that $D$ is defined by $\rho>0$, where $\rho\in C^\infty(\overline{D})$ is a defining function for the boundary $\partial D$, i.e., $\rho|_{\partial\Omega}=0$ and $d\rho|_{\partial D}\neq 0$. Fefferman proposed (\cite{Feff76}, \cite{Feff79}) investigating the biholomorphic geometry of $D$ (e.g., the Bergman kernel) and the CR geometry of the boundary $M=M^{2n-1}:=\partial D$ via invariants obtained by restricting to a class of special defining functions $\rho$ normalized by $J(\rho)=1+O(\rho^{n+1})$, where $J$ is the complex Monge-Amp\`ere operator
\beq\label{Jop}
J(u):=(-1)^{n}\det
\begin{pmatrix}
u&u_{\bar z_k}\\
u_{z_j} & u_{z_j \bar z_k}
\end{pmatrix}.
\eeq
Fefferman showed in \cite{Feff76} that such a smooth defining function exists, and that it is unique mod $O(\rho^{n+2})$. A defining function $\rho$ satisfying this normalization is called a {\it Fefferman defining function}.  The work of Cheng--Yau \cite{ChengYau80}, combined with the subsequent work of Lee--Melrose \cite{LeeMelrose82}, shows that the Dirichlet problem
\beq\label{JDir}
J(u)=1,\quad  u|_M=0
\eeq
has a unique non-negative solution $u\in C^\infty(D)\cap C^{n+2-\epsilon}(\overline{D})$, which has an asymptotic expansion of the form
\beq\label{uexp}
u\sim \rho\sum_{k=0}^\infty \eta_k(\rho^{n+1}\log\rho)^k,\quad \eta_k\in C^\infty(\overline{D}).
\eeq
One observes that in the case of the unit ball $D=\bB^n\subset \bC^n$, the solution $u$ to \eqref{JDir} coincides with the standard defining function $\rho=1-\norm z\norm^2$; thus, there is no singularity in this case and we can take $\eta_k=0$ for $k\geq 1$. Graham (\cite{Graham87full}, \cite{Graham87}) showed that the boundary value problem \eqref{JDir} can be solved formally near a point on $M$, yielding a family of formal solutions $u$ of the form \eqref{uexp} that depend on one additional parameter function (which adds a condition on a normal derivative, to complete the Cauchy data for the local problem \eqref{JDir}). Moreover, the functions $\eta_k$ for $k\geq 1$, which make up the singularity of the solution, are uniquely determined mod $O(\rho^{n+1})$ by the local CR geometry of the boundary $M$ only (independent of the additional parameter function and choice of Fefferman defining function $\rho$). In particular, the functions $b\eta_k:=\eta_k|_M$ are uniquely determined smooth functions on the boundary $M$. Indeed, Graham proved that the $b\eta_k$ are local CR invariants of weight $(n+1)k$; (see, e.g., \cite{Graham87full}, \cite{Graham87} for the formal definition of this notion; see also below). Graham also showed, on the one hand, that if $b\eta_1=0$ on $M$, then all the functions $\eta_k$, $k\geq 1$, vanish to infinite order on $M$ and, hence, there is a smooth function ($\rho\eta_0$ in the expansion \eqref{uexp}) that agrees with the solution $u$ to infinite order (as in the case of the unit ball $\bB^n$); on the other hand, he showed that there are strictly pseudoconvex hypersurfaces $M\subset \bC^n$ such that $b\eta_1=0$, but $M$ is not locally spherical (i.e., not locally equivalent to the sphere $\partial\bB^n$). In other words, information about the function $b\eta_1$ {\it locally} on the boundary $M$ determines completely the singularity (mod $O(\rho^{n+1})$) of the solution $u$ near $M$, but {\it does not} determine the CR geometry of $M$. It may still be the case, however, that information about $b\eta_1$ {\it globally} may determine the CR geometry of $M$. The main result in this paper is a result along these lines in $\bC^2$ for domains with transverse symmetry. The function $b\eta_1$ is sometimes referred to as the {\it obstruction function} (cf., \cite{Hirachi14}).

The solution $u$ to \eqref{JDir} is intimately related to the Bergman and Szeg\H o kernels of the domain $D\subset\bC^n$; these are the reproducing kernels of the holomorphic functions, respectively the boundary values of holomorphic functions, in $L^2(D)$ and $L^2(M,\sigma)$, where $\sigma$ denotes some choice of surface element on $M$. We recall (\cite{Fefferman74}, \cite{BoutetSjostrand76}) that the Bergman and Szeg\H o kernels, $K_B(Z)$ and $K_S(Z)$, of $D$ on the diagonal have asymptotic expansions of the form
\beq\label{BandS}
K_B\sim \frac{\phi_B}{\rho^{n+1}}+\psi_B\log\rho,\quad K_S\sim \frac{\phi_S}{\rho^{n}}+\psi_S\log\rho, \quad \phi_B,\phi_S,\psi_B,\psi_S\in C^\infty(\overline{D}),
\eeq
in terms of a Fefferman defining function $\rho$. To make the Szeg\H o kernel $K_S$ biholomorphically invariant, we have chosen here the invariant surface element on $M=\partial \Omega$ as in \cite{HKN93} (see also \cite{Lee88}, \cite{Hirachi93}). The functions $\phi_B$, $\phi_S$ are determined mod $O(\rho^{n+1})$, $O(\rho^{n})$, respectively, and $\psi_B$, $\psi_S$ are determined up to infinite order. In the special case $D=\bB^n$, both $\psi_B$ and $\psi_S$ vanish identically (and $\phi_B$, $\phi_S$ are constants). The Ramadanov Conjecture \cite{Ramadanov81} predicts the converse (for the Bergman kernel, although the conjecture has also been stated for the Szeg\H o kernel): {\it If $\psi_B$ (or $\psi_S$) vanishes to infinite order at $M$, then $D$ is biholomorphic to the unit ball $\bB^n$.} The conjecture is known to be true in $\bC^2$ (at least if $D$ is assumed simply connected with connected boundary) due to the work of Graham (who attributes it to Burns), but is still open for $n\geq 3$. The solution in $\bC^2$ hinges on the result, proved in \cite{Graham87}, that the asymptotic expansion of $\psi_B$ in this case is as follows:
\beq\label{K_Bexp}
\psi_B=a_1 \eta_1+a_2Q\rho+O(\rho^2), a_1,a_2\in\bC\setminus\{ 0\},
\eeq
where $Q$ is E. Cartan's 6th order (umbilical) invariant. Thus, if $\psi_B$ is $O(\rho^2)$, then we may conclude that $Q=0$ (by using the result of Graham that $b\eta_1=0\implies \eta_1=0$), which is well known \cite{Cartan33} to imply that $M$ is locally spherical. This proves that $D$ is biholomorphic to $\bB^2$, if $D$ is assumed simply connected with connected boundary, by the Riemann mapping theorem of Chern--Ji \cite{ChernJi96}. In subsequent work, Nakazawa \cite{Nakazawa94} (see also Boichu--Coeur\'e \cite{Boichu83}) proved that for complete Reinhardt domains, it suffices to assume that $\psi_B|_M=b\eta_1=0$ to conclude that $D$ is biholomorphic to $\bB^2$; the latter result is an example of a situation where global vanishing of $b\eta_1$ forces $M$ to be locally spherical. Analogous results hold for the Szeg\H o kernel, normalized by the invariant surface element on $M$, in view of the expansion of $\psi_S$ for $n=2$ due to Hirachi--Komatsu--Nakazawa \cite{HKN93}:
\beq\label{K_Sexp}
\psi_S=c_1 (b\eta_1)\rho+c_2Q\rho^2+O(\rho^3), c_1,c_2\in\bC\setminus\{ 0\}.
\eeq
The reader is also referred to subsequent work on CR invariants and the expansions of the Bergman and Szeg\H o kernels by, e.g., Bailey--Eastwood--Graham \cite{BaileyEtAl94}, Hirachi \cite{Hirachi00}, \cite{Hirachi06}, and others.

In this note, we shall consider the case $n=2$, i.e., bounded strictly pseudoconvex domains $D\subset \bC^2$. The boundary $M=M^3=\partial D$ is then a compact three dimensional strictly pseudoconvex CR manifold. As illustrated by the result of Graham mentioned above, the vanishing of $b\eta_1$ on an open subset $U\subset M$ does not imply that $U$ is locally spherical in general. Our main result, however, is that if $D$ has transverse symmetry, then the vanishing of $b\eta_1$ globally on $M$ implies that $M$ is locally spherical.

We recall that $D$ has {\it transverse symmetry} if there is a 1-parameter family of biholomorphisms of $\overline{D}$ such that its infinitesimal generator is transverse to the CR tangent space on the boundary $M:=\partial D$. Examples include {\it circular} domains, i.e., those for which $Z\in D$ if and only if the whole circle $T_Z:=\{e^{it}Z\colon t\in \bR\}$, is contained in $D$. In particular any Reinhardt domain is circular and, hence, has transverse symmetry. Our main result is the following:

\begin{theorem}\label{MainThmFeff} Let $D\subset \bC^2$ be a smooth bounded strictly pseudoconvex domain, and assume further that $D$ has transverse symmetry. Then, $b\eta_1=0$ on $M:=\partial D$ if and only if $M$ is locally spherical. If $D$ is simply connected and $M$ connected, then $b\eta_1=0$ on $M$ if and only if $D$ is biholomorphic to the unit ball $\bB^2\subset \bC^2$.
\end{theorem}

In view of the expansions \eqref{K_Bexp} and \eqref{K_Sexp} of the log terms in the Bergman and Szeg\H o kernels, we obtain the following direct corollaries of Theorem \ref{MainThmFeff}:

\begin{corollary}\label{MainThmBerg} Let $D\subset \bC^2$ be a smooth bounded strictly pseudoconvex domain, and assume further that $D$ has transverse symmetry. Let $K_B$ denote the Bergman kernel of $D$ with asymptotic expansion given by \eqref{BandS}. Then, the log term $\psi_B|_M=0$ on $M:=\partial D$ if and only if $M$ is locally spherical. If $D$ is simply connected and $M$ connected, then $\psi_B|_M=0$ on $M$ if and only if $D$ is biholomorphic to the unit ball $\bB^2\subset \bC^2$.
\end{corollary}

\begin{corollary}\label{MainThmSzego} Let $D\subset \bC^2$ be a smooth bounded strictly pseudoconvex domain, and assume further that $D$ has transverse symmetry. Let $K_S$ denote the Szeg\H o kernel of $D$, normalized by the invariant surface element on $M:=\partial D$, with asymptotic expansion given by \eqref{BandS}. Then, the log term $\psi_S=O(\rho^2)$ on $M$  if and only if $M$ is locally spherical. If $D$ is simply connected and $M$ connected, then $\psi_S=O(\rho^2)$ on $M$ if and only if $D$ is biholomorphic to the unit ball $\bB^2\subset \bC^2$.
\end{corollary}

We should briefly mention the role of the choice of surface element on $M$ in the Szeg\H o kernel $K_S$, since Corollary \ref{MainThmSzego} in the special case of complete circular domains appears similar to a result in \cite{LuTian04}. For each choice of contact form $\theta$ on $M$, one obtains a Szeg\H o kernel $K^\theta_S$ corresponding to the surface element $\sigma[\theta]:=\theta\wedge d\theta$ on $M$. The invariant surface element (\cite{HKN93}; see also \cite{Lee88} and \cite{Hirachi93}) corresponds to the unique choice of $\theta=\theta_0$ such that
\beq\label{invsurf}
\sigma[\theta_0]\wedge d\rho=J(\rho)^{1/(n+1)}dV,\quad dV=\frac{1}{-2i}\bigwedge_{j=1}^ndz_j\wedge d\bar z_j.
\eeq
It is shown in \cite{HKN93} that for the invariant surface element on $M$ in $\bC^2$, it holds that $\psi_S|_M=0$, where $\psi_S=\psi_S^{\theta_0}$. This leads to the form of the expansion indicated in \eqref{K_Sexp}. Hirachi further showed \cite{Hirachi93} that in fact
\beq
\psi_S^\theta|_M=\frac{1}{24\pi}(\Delta_b R-2\im A_{11;}{}^{11}),
\eeq
where $\Delta_b$, $R$, $A_{11}$ are the sublaplacian, the Tanaka--Webster scalar curvature, and the Tanaka-Webster torsion, respectively, of the pseudohermitian structure corresponding to $\theta$ (see \cite{Webster78}). Moreover, he showed that if $M$ has transverse symmetry then $\psi_S^\theta|_M=0$ if and only if $\theta=e^{2f}\theta_0$ for some pluriharmonic function $f$ on $M$.

In some situations, there may also be natural choices of surface element on $M$, other than the invariant one. For instance, if $D$ is the unit disk bundle in a negative holomorphic line bundle $L^*$ over a Riemann surface $X$, then a natural surface element is $\sigma=\omega\wedge dt$, where $-\omega$ is the K\"ahler form on $X$ obtained from the curvature form of $L^*$ and $t\mapsto (z,e^{it}\ell)$ the circle action on $M:=\partial D$. The Szeg\H o kernel corresponding to this surface element on the disk bundle (also over higher dimensional K\"ahler manifolds) has been considered by many authors. We mention here only \cite{Tian90}, \cite{Zelditch98}, \cite{Catlin99}, \cite{LuTian04}, and refer to these papers for further references. In particular, in \cite{LuTian04} the analog of Ramadanov's Conjecture above was considered for the Szeg\H o kernel in a disk bundle $D$ over the complex projective plane $\bP^1$ corresponding to the surface element $\sigma[\theta]=\omega\wedge dt$. The result in this case is that if the log term $\psi_S^{\theta}$ vanishes on $M$, then $\omega$ is the Fubini-Study form on $\bP^1$ (up to an automorphism $\bP^1\to \bP^1$), which is equivalent to the statement that $D$ is the blow-up of the origin in the unit ball $\bB^2\subset \bC^2$ (up to an automorphism). We wish to emphasize that while a complete circular domain in $\bC^2$, a special case of the main result in this paper, is the blow-down of a disk bundle over $\bP^1$, the assumptions in Corollary \ref{MainThmSzego} and in \cite{LuTian04} are different, as the Szeg\H o kernels are taken with respect to {\it a priori} different surface measures.

The paper is organized as follows. In Section \ref{S:general} we establish a correspondence  between the obstruction function $b\eta_1$ and the classical invariants of E. Cartan and Chern--Moser. In Section \ref{S:hlb} we consider the special case of disk bundles in (duals of) positive holomorphic line bundles. The calculations in this case are classical, and requires no prior experience with pseudohermitian geometry. In the subsequent section, we explain how the calculation in a CR manifold with transverse symmetry can be reduced to that in the disk bundle case. The final section \ref{S:proof} is then devoted to the proof of the main result, Theorem \ref{MainThmFeff}.

\section{The weight $\kappa=3$ invariant}\label{S:general}

Let $M=M^3$ be a three dimensional strictly pseudoconvex CR manifold, which we shall always assume to be locally embeddable as a real hypersurface in $\bC^n$, for some $n$. Recall that a CR invariant of a positive weight $\kappa$ is a polynomial in "data" associated with the CR structure that transforms under CR diffeomorphisms by scaling with the Jacobian of the diffeomorphism to the power $2\kappa/3$ (see, e.g., \cite{Graham87full}, \cite{Graham87}). Typical "data" are the covariant derivatives of the components of the Tanaka--Webster curvature and torsion, in which case CR invariants are special cases of pseudohermitian invariants (see e.g., \cite{Hirachi93}). Another approach is to use the coefficients $A^j_{kl}$ in the Chern--Moser normal form \cite{CM74} in (local or formal) coordinates $(z,w)\in \bC^2$:
\beq\label{CMnormal}
\im w=|z|^2+\sum_{k,l\geq 2}\sum_{j=1}^\infty A^j_{kl}z^k\bar z^l(\re w)^j;\quad A^j_{22}=A^j_{23}=A^j_{33}=0,\ j=0,1,2,\ldots.
\eeq
It was shown by R. C. Graham \cite{Graham87} that there are no (nontrivial) CR invariants of weight $\kappa=1,2$, and that the space of CR invariants of weight $3$ and $4$, respectively, is  1-dimensional and spanned by $A^0_{44}$ and $|A^0_{24}|^2$. It is well known that the coefficient $A^0_{24}$, while not a CR invariant of a positive weight in the sense of Graham (but rather of a "complex weight" of type (2,4)), represents E. Cartan's "6th order invariant" $Q=Q^1{}_{\bar 1}$ obtained in his solution to the CR equivalence problem for three dimensional strictly pseudoconvex CR manifolds \cite{Cartan33}. We shall show here that the weight 3 invariants, spanned in the Chern--Moser setup by $A^0_{44}$, can be also represented by a second order covariant derivative of Cartan's invariant $Q$. To explain this, we recall here E. Cartan's solution to the equivalence problem, following the exposition of Jacobowitz \cite{JacobowitzBook} (but with slightly different notation).

As above, let $M=M^3$ be a 3-dimensional strictly pseudoconvex CR manifold. There is an 8-dimensional bundle $\pi\colon B\to M$ and an invariantly defined coframe
\beq\label{Ccoframe}
\{\Omega,\Omega^1,\Omega^{\bar 1},\Omega^2, \Omega^{\bar 2}, \Omega^3, \Omega^{\bar 3},\Omega^4\},
\eeq
with $\Omega,\Omega^4$ real-valued, $\Omega^{\bar l}:=\overline{\Omega^l}$ for $l=1,2,3$, such that the following structure equations hold:
\beq\label{CStructure}
\begin{aligned}
d\Omega &=i\Omega^1\wedge\Omega^{\bar 1}-\Omega\wedge(\Omega^2+\Omega^{\bar 2})\\
d\Omega^1 &=-\Omega^1\wedge\Omega^2-\Omega\wedge\Omega^3\\
d\Omega^2 &= 2i\Omega^1\wedge\Omega^{\bar 3}+i\Omega^{\bar 1}\wedge \Omega^3-\Omega\wedge\Omega^4\\
d\Omega^3 &=-\Omega^1\wedge\Omega^4-\Omega^{\bar 2}\wedge\Omega^3-Q\Omega\wedge\Omega^{\bar 1}\\
d\Omega^4 &=i\Omega^3\wedge\Omega^{\bar 3}-(\Omega^2+\Omega^{\bar 2})\wedge\Omega^4-S\Omega\wedge\Omega^1-\bar S\Omega\wedge\Omega^{\bar 1},
\end{aligned}
\eeq
where $Q$ (Cartan's invariant) and $S$ are functions on $B$. Cartan showed that $M$ is spherical near a point $p\in M$ if and only if $Q$ vanishes over a neighborhood of $p$ in $M$. We may now construct new invariant functions on $B$ by taking "covariant" differentiations of the invariant functions $Q$ and $R$ with respect to the invariant coframe \eqref{Ccoframe}, e.g.,
\beq\label{Cdiff}
dQ=Q_{;0}\Omega+\sum_{l=1}^3(Q_{;l}\Omega^l+Q_{;\bar l}\Omega^{\bar l})+Q_{;4}\Omega^4.
\eeq
An easy calculation, differentiating the structure equation for $d\Omega^3$, reveals that $$
\bar S=Q_{1},
$$
and hence repeated covariant differentiation of $Q$ will yield all invariant functions. We claim that $Q_{;11}$ is a CR invariant of weight $\kappa=3$. We will first need to explain how a choice of contact form $\theta$ near a point $p\in M$ leads to a a polynomial expression in the Chern-Moser normal form coefficients in \eqref{CMnormal}. In order to carry this out, we shall compute $Q_{;11}$ in a special local coordinate system on $B$, following the book by Jacobowitz \cite{JacobowitzBook}. Let $\theta$ be a contact form on $M$, $x=(z,t)\in U\subset \bC\times \bR$ a local chart on $M=M^3$ such that $\{\theta,\theta^1\}$, with $\theta^1:=dz$, defines the CR structure on $M$. We shall normalize the choice of contact form $\theta$ so that the Levi form of $M$ with respect to $\theta^1=dz$ is one, i.e.,
\beq\label{normtheta}
d\theta=i\theta^1\wedge\theta^{\bar 1}+b\theta\wedge\theta^1+\bar b\theta\wedge\theta^{\bar 1},
\eeq
for some function $b=b(x)$ on $M$. As in \cite{JacobowitzBook}, we may then choose coordinates $(x,\lambda,\mu,\rho)\in U\times \bC\times\bC\times\bR$ on $\pi^{-1}(U)\subset B$ such that
\beq\label{Cforms}
\begin{aligned}
\Omega &=|\lambda|^2\theta\\
\Omega^1 &=\lambda(\theta^1+\mu\theta)\\
\Omega^2 &=\frac{d\lambda}{\lambda}+A\theta^1+B\theta^{\bar 1}+C\theta\\
\Omega^3 &=\frac{1}{\bar\lambda}\left(d\mu+D\omega^1+E\theta^{\bar 1}+F\theta\right)\\
\Omega^4 &=\frac{1}{|\lambda|^2}\left(d\rho+\frac{i}{2}(\mu d\bar\mu-\bar \mu d\mu)+H\theta^1+\bar H\theta^{\bar 1}+G\theta\right),
\end{aligned}
\eeq
where $A,B,C,D,E,F,G,H$ are functions in $x,\lambda,\mu,\rho$ explicitly computed in \cite{JacobowitzBook}. To compute $Q_{;11}$, we shall only require the expressions for $A,B,E$, which we reproduce here
\beq\label{ABE}
A=-(b+2i\bar \mu),\quad B=-i\mu,\quad E=-\mu(\bar b-i\mu).
\eeq
Next, we recall from \cite{JacobowitzBook} that in the coordinates $(x,\lambda,\mu,\rho)$,
\beq\label{Q=}
Q=\frac{r}{\lambda\bar\lambda^3},\quad r=r(x).
\eeq
We let $L_1$ be the $(1,0)$ vector field and $T$ the transversal vector field in $U\subset M$ such that the frame $\{T,L_1,L_{\bar 1}\}$ is dual to the coframe $\{\theta,\theta^1,\theta^{\bar 1}\}$ and compute
\beq
dQ =r\left(-\frac{d\lambda}{\lambda^2\bar\lambda^3}-3\frac{d\bar \lambda}{\lambda\bar\lambda^4}\right)+\frac{1}{\lambda\bar\lambda^3}(L_1r\, \theta^1+L_{\bar 1}r \, \theta^{\bar 1}+Tr\, \theta)
\eeq
and obtain, using \eqref{Cforms},
\beq
dQ= \frac{1}{\lambda\bar\lambda^3}\left(-r\Omega^2+rA\frac{\Omega^1}{\lambda}-3r\Omega^{\bar 2}+3r\bar B\frac{\Omega^1}{\lambda}+L_1r\, \frac{\Omega^1}{\lambda}\right )\mod \Omega,\Omega^{\bar 1}.
\eeq
Consequently, by using also \eqref{ABE}, we conclude
\beq
Q_{;1}=\frac{1}{\lambda^2\bar\lambda^3}(L_1r+r(-b+i\bar \mu))=\frac{1}{\lambda^2\bar\lambda^3}(L_1r-rb+ir\bar \mu).
\eeq
We differentiate again and obtain
\begin{multline}
dQ_{;1} =(L_1r-rb+ir\bar \mu)\left(-2\frac{d\lambda}{\lambda^3\bar\lambda^3}-3\frac{d\bar \lambda}{\lambda^2\bar\lambda^4}\right)+\frac{ir}{\lambda^2\bar\lambda^3}d\bar\mu\\ +
\frac{1}{\lambda^2\bar\lambda^3}(L_1^2 r-L_1(rb)+i L_1r\, \bar\mu)\theta^1 \mod \theta,\theta^{\bar 1}.
\end{multline}
Using again \eqref{Cforms}, we obtain
\begin{multline}
dQ_{;1} =\frac{1}{\lambda^2\bar\lambda^3}\bigg\{ (L_1r-rb+ir\bar \mu)\left(-2\Omega^2+2A\frac{\Omega^1}{\lambda}-3\Omega^{\bar 2}+3\bar B\frac{\Omega^{1}}{\lambda}\right)+ir\left(\lambda\Omega^{\bar 3}-\bar E\frac{\Omega^1}{\lambda}\right)
\\ +
(L_1^2 r-L_1(rb)+i L_1r\, \bar\mu)\frac{\Omega^1}{\lambda} \bigg\}
\mod \Omega,\Omega^{\bar 1}.
\end{multline}
Thus, we obtain
\beq\label{Q11-1}
Q_{;11}=\frac{1}{\lambda^3\bar\lambda^3}(L_1^2r-L_1(rb)+iL_1r\, \bar \mu + (L_1r-rb+ir\bar \mu)(2A+3\bar B)-ir \bar E).
\eeq
Applying again \eqref{ABE}, we find that
\beq
(L_1r-rb+ir\bar \mu)(2A+3\bar B)-ir \bar E=-2(L_1r)b+2rb^2-i(L_1r)\bar \mu,
\eeq
and hence we obtain from \eqref{Q11-1}
\beq\label{Q11-2}
Q_{;11}=\frac{1}{\lambda^3\bar\lambda^3}\left(L_1^2r-3(L_1r)b+r(2b^2-L_1b)\right).
\eeq
We note in particular that $Q_{;11}$ is of the form
\beq\label{Q11-3}
Q_{;11}=\frac{s(x)}{|\lambda|^6},
\eeq
where, in the special fiber coordinates $(\lambda,\mu,\rho)$ corresponding to the choice of $\{\theta,\theta^1\}$ above,
\beq \label{s(x)}
s(x)=L_1^2r(x)-3(L_1r(x))b(x)+r(x)(2b(x)^2-L_1b(x)).
\eeq
Now, we note that if $(z,w)\in \bC^2$ are formal Chern--Moser coordinates for $M$ centered at $p=(0,0)$ so that $M$ is formally given by an equation of the form \eqref{CMnormal},
which we write temporarily as
\beq\label{imw=phi}
\im w=\Phi(z,\bar z,\re w), \quad \Phi(z,\bar z,t):= |z|^2+\sum_{k,l\geq 2}\sum_{j=1}^\infty A^j_{kl}z^k\bar z^l t^j,
\eeq
then we may choose $x=(z,t)$ with $t:=\re w$ as local coordinates, and  we may use the contact form (cf.\ \cite{BER99a})
\beq\label{CMcontact}
\theta= \left(\frac{\partial}{\partial z}\frac{\Phi_{\bar z}}{1+i\Phi_t}-\frac{\partial}{\partial \bar z}\frac{\Phi_{z}}{1-i\Phi_t}\right)^{-1}\left (dt-i\frac{\Phi_{z}}{1-i\Phi_t} dz +i\frac{\Phi_{\bar z}}{1+i\Phi_t}d\bar z\right)
\eeq
in the calculations carried out above. We obtain an evaluation of $Q_{;11}$ on the contact form $\theta$ in \eqref{CMcontact} by evaluating \eqref{Q11-2} at $\lambda=1$; we denote this evaluation by $Q_{;11}[\theta]$. We now note that by the form of $\Phi(z,\bar z,t)$ given by \eqref{imw=phi},
\beq\label{Q11at0}
\left(\frac{\partial}{\partial z}\frac{\Phi_{\bar z}}{1+i\Phi_t}-\frac{\partial}{\partial \bar z}\frac{\Phi_{z}}{1-i\Phi_t}\right)\bigg|_{(z,t)=(0,0)}=1,
\eeq
and, hence, it follows from \eqref{Q11-2} that $Q_{;11}[\theta]$, evaluated at $p=(0,0)$, is a polynomial in the Chern--Moser coefficients $A^j_{kl}$. In fact, our main result in this section is the following:

\begin{theorem}\label{Thm-wt3} There is a universal constant $c\neq 0$, such that
\beq\label{Q11}
Q_{;11}[\theta]=cA^0_{44},
\eeq
where $\theta$ is given by \eqref{CMcontact} and \eqref{imw=phi}, and $A^0_{44}$ is the $z^4\bar z^4$ coefficient in the Chern--Moser normal form \eqref{CMnormal}.
\end{theorem}

\begin{proof} To prove Theorem \ref{Thm-wt3}, we shall show that $Q_{;11}$ is a (nontrivial) CR invariant of weight $3$. In view of Theorem 2.1 in \cite{Graham87}, which states that the space of CR invariants of weight 3 is 1-dimensional and spanned by $A^0_{44}$, we can then conclude that there exists a constant $c$ such that \eqref{Q11} holds.  To prove that $c\neq0$, it suffices to show that $Q_{;11}$ is not zero for some CR manifold $M$. We shall in fact show (Corollary \ref{CorKzzzz=0} below) that for unit circle bundles $M$ over compact Riemann surfaces, the identity $Q_{;11}=0$ characterizes those that are locally spherical. Since there clearly are such $M$ (these include all boundaries of complete circular domains) that are not locally spherical, we deduce that $c\neq 0$.

Recall now (e.g., \cite{Graham87}, \cite{Hirachi93}) that a pseudohermitian invariant $I(\theta)$ (computed as a polynomial in covariant derivatives of the curvature and torsion of the pseudohermitian structure given by a contact form $\theta$, or as a polynomial in the Chern--Moser coefficients $A^j_{kl}$) is a CR invariant of weight $\kappa$ if for any other contact form $\tilde\theta=e^u\theta$, $u\in\bC^\infty(M)$, we have
$$
I(\tilde \theta)=e^{-\kappa u}I(\theta).
$$
Since $Q_{;11}$ is an invariant function on the bundle $B$ of the form \eqref{Q11-3}, it is clear, by taking $|\lambda|^2=e^{u}$, that $Q_{;11}$ is a CR invariant of weight $\kappa=3$. As mentioned above, it follows from Corollary \ref{CorKzzzz=0} that this invariant is nontrivial. This completes the proof of Theorem {Thm-wt3}.
\end{proof}

We may also reformulate the result of the discussion above as follows:

\begin{theorem}\label{Thm-inv}
The invariant function $Q_{;11}$ is a nontrivial CR invariant of weight $\kappa=3$.
\end{theorem}

\section{Circle bundles over Riemann surfaces}\label{S:hlb}

Let $X$ be a Riemann surface (complex manifold of dimension 1) and $\pi\colon L\to X$ a positive holomorphic line bundle, with $(\cdot,\cdot)$ a positively curved metric on $L$, and endow $X$ with the K\"ahler metric $ds^2$ induced by the curvature of $L$. Let $L^*$ be the dual line bundle, equipped with the dual metric, and $D$ the unit disk bundle in $L^*$. It is well known that $D$ is a strictly pseudoconvex domain. We shall mainly be interested in its boundary $M:=\partial D$, the unit circle bundle in $L^*$, which is then a strictly pseudoconvex, three dimensional CR manifold given by
$$
M=\{(x,\ell^*)\in L^*\colon |\ell^*|^2_x=1\}.
$$
If $s_0\colon U\subset X\to L$ is a nonvanishing local holomorphic section, then in the induced local trivialization $L^*|_U\cong U\times\bC$ with coordinates $(z,\tau)\in U\times \bC$, the three dimensional CR manifold $M$ is given by
\begin{equation}\label{Meq}
|\tau|^2h(z,\bar z)^{-1}=1,
\end{equation}
where $h(z,\bar z)=|s_0|_z^2$. The assumption that the curvature of $L$ is positive means that
\begin{equation}\label{Lcurv}
i \Theta:=-i \partial\bar\partial \log h>0.
\end{equation}
If we use polar coordinates $\tau=re^{it}$ in the fibers and $(z,t)\in \bC\times \bR$ as local coordinates on $M$, then
\begin{equation}\label{thetahat}
\hat \theta=dt+\frac{i}{2}(\partial \log h-\bar\partial \log h)
\end{equation}
is a contact form on $M$ that is compatible with the CR structure, and
\begin{equation}
d\hat \theta=\frac{i}{2}(\bar\partial \partial \log h-\partial\bar \partial \log h)=-i\partial\bar \partial \log h.
\end{equation}
We shall use the notation
\begin{equation}\label{Dnotation}
D:=\frac{\partial}{\partial z},\quad \Delta:=4D\bar D,
\end{equation}
so that
$$
d\hat\theta=-iD\bar D\log h\,  dz\wedge d\bar z=ia^{-1}\,  dz\wedge d \bar z
$$
where $a=a(z,\bar z)$ is the function
\begin{equation}\label{a}
a:=(-D\bar D\log h)^{-1}=\left(-\frac{1}{4}\Delta \log h\right)^{-1}>0.
\end{equation}
Thus, with
\begin{equation}\label{theta-theta^1}
\theta:=a\hat\theta, \theta^1:=dz
\end{equation}
we have
\begin{equation}\label{dtheta_0}
d\theta=i\, \theta^1\wedge \theta^{\bar 1} -\theta\wedge\frac{da}{a}.
\end{equation}
In other words, we can use $x=(z,t)$ and the forms in \eqref{theta-theta^1} to set up Cartan's bundle $B$ as described in the previous section. In this case, the function $b=b(x)$ in \eqref{normtheta} is independent of the circle coordinate $t$, and
\beq\label{b=Da/a}
b=b(z,\bar z)=-Da/a=-D \log a(z,\bar z)=D\log\, (-D\bar D\log h).
\eeq
Next, recall the invariant function $Q$ (Cartan's tensor) in \eqref{Q=}. The direct computation in \cite{JacobowitzBook}  shows that (see pp. 126 and 140 in \cite{JacobowitzBook}), in the chosen coordinate system $x=(z,t)$, the function $r=r(x)$ in \eqref{Q=} is a function of $z$ alone, explicitly computed from the function $b=b(z,\bar z)$ in \eqref{b=Da/a}. In fact, $r$ is obtained by applying a third order differential operator to $\bar b$ (see \cite{JacobowitzBook}, eq. (47) on p. 126):
\begin{equation}\label{r}
r=\frac{1}{6}(\bar D^2D \bar b-3\bar b D\bar D \bar b +2\bar b^2\bar D\bar b -D\bar b\bar D\bar b).
\end{equation}
Recall that the Riemann surface $X$ is calibrated by the positive holomorphic line bundle $L$, i.e., equipped with the K\"ahler metric induced by the curvature of metric $(\cdot,\cdot)$ on $L$. In the local chart $x=(z,t)$ in $U\subset X$, we then have
$$
ds^2=e^{2\phi}|dz|^2,\quad 2\phi:=\log(-D\bar D\log h)=-\log a.
$$
We shall denote by $K$ the Gauss curvature of $ds^2$,
\begin{equation}\label{K}
K:=-e^{-2\phi}\Delta\phi=-4e^{-2\phi}D\bar D\phi.
\end{equation}
For a smooth, real-valued function $f$, we shall denote by $f_{;z}$, $f_{;\bar z}$, $f_{z^2}:=f_{;zz}$, \ldots, $f_{;z^k\bar z^k}$, etc., the repeated covariant derivatives with respect to $z$ (in the $(1,0)$ direction) and $\bar z$ (in the $(1,0)$ direction) in the unitary coframe $e^\phi dz$; i.e., since the (dual) connection form in this case equals $-(\partial-\bar\partial)\phi$ (e.g., \cite{GHbook}, p. 77), we have $f_{;z}=e^{-\phi}Df$, $f_{;\bar z}= e^{-\phi}\bar Df$ and inductively
\begin{equation}\label{Kzz}
\begin{aligned}
f_{;z^k \bar z^lz} &=e^{-\phi}(Df_{;z^k\bar z^l}+(l-k)(D\phi)f_{;z^k\bar z^l})=e^{(k-l-1)\phi}D(e^{(l-k)\phi}f_{;z^k\bar z^l})\\
f_{;z^k \bar z^l\bar z} &=e^{-\phi}(\bar Df_{;z^k\bar z^l}+(k-l)(\bar D\phi)f_{;z^k\bar z^l})=e^{(l-k-1)\phi}\bar D(e^{(k-l)\phi}f_{;z^k\bar z^l})
\end{aligned}
\end{equation}

\begin{theorem}\label{QisGauss} The invariant functions $Q$ and $Q_{;11}$ are related to the Gauss curvature $K$ of $(X,ds^2)$ via:
\beq\label{GaussQs}
Q=-\frac{e^{4\phi}}{12}\ \frac{K_{;\bar z\bar z}}{\lambda\bar\lambda^3},\quad Q_{;11}= -\frac{e^{6\phi}}{12}\ \frac{K_{;\bar z\bar zz z}}{|\lambda\bar\lambda|^3}.
\eeq
\end{theorem}

\begin{proof} The first identity in \eqref{GaussQs} was already observed in \cite{EDumb15}, Proposition 4.1 (but note that in that paper the complex conjugate of $Q$ was considered). The proof is a direct computation of $K_{;\bar z\bar z}$, using the expressions $$
K=-4e^{-2\phi}\bar D D\phi,\quad K_{;\bar z\bar z}=\bar D(e^{-\phi}(e^{-\phi}\bar DK))=e^{-2\phi}(\bar D^2K-2(\bar D\phi)\bar DK),
$$
and comparing the result with \eqref{r}, recalling that $\bar b=2\bar D\phi$. To obtain the second identity in \eqref{GaussQs}, we recall that $Q_{;11}$ is of the form \eqref{Q11-3}, where $s$ in this case is a function of $z$ alone, $s=s(z,\bar z)$, given by \eqref{s(x)}, which becomes
\beq \label{s(z)}
s=D^2r-3(Dr)b+r(2b^2-Db).
\eeq
We also have
\beq\label{Kzzzz}
\begin{aligned}
K_{;\bar z\bar z z z} &=e^{-2\phi}D(e^{\phi}(e^{-3\phi}D(e^{2\phi} K_{;\bar z\bar z})))\\
&= e^{-2\phi}D(e^{-2\phi}D(e^{2\phi} K_{;\bar z\bar z})).
\end{aligned}
\eeq
By the first identity in \eqref{GaussQs}, we have
$$
K_{;\bar z\bar z}=-12e^{-4\phi}r
$$
and, hence, by \eqref{Kzzzz}
\beq
-\frac{1}{12}K_{;\bar z\bar z z z}=e^{-2\phi}D(e^{-2\phi}D(e^{2\phi} (e^{-4\phi}r)))=e^{-2\phi}D(e^{-2\phi}D(e^{-2\phi}r))
\eeq
By expanding this, comparing with \eqref{s(z)} and recalling $b=2D\phi$, we conclude that the second identity in \eqref{GaussQs} holds.
\end{proof}

\begin{remark} We note that there is a similar local divergence form in general for $s(x)$ in \eqref{s(x)} provided we can find a function $u$ such that $b=L_1u$. It can be verified by direct calculation that
\beq\label{generaldiv}
s=e^{2u}L_1(e^{-u}L_1(e^{-u}r)).
\eeq
This fact is used in the next section.
\end{remark}

We may now prove the following result, which has been alluded to above.

\begin{corollary}\label{CorKzzzz=0} Let $X$ be a compact Riemann surface, $(L,h)\to X$ a holomorphic line bundle with positively curved metric $h=(\cdot,\cdot)$, and $D$ the unit disk bundle in the dual line bundle $(L^*,h^{-1})\to X$. Then CR invariant function $Q_{;11}$ vanishes on the unit circle bundle $M:=\partial D$ if and only if $M$ is locally spherical.
\end{corollary}

\begin{proof}
By Theorem \ref{QisGauss}, $Q_{;11}$ vanishes on $M$ if and only if $K_{;\bar z\bar z z z}=0$ on $X$. Calabi proved (\cite{Calabi82}, Lemma on p. 273) that this implies (and of course follows from) $K_{;\bar z\bar z}=0$ on $M$, which implies (and follows from), again by Theorem \ref{QisGauss}, that Cartan's 6th order invariant $Q$ vanishes on $M$. The latter is well known to be equivalent to $M$ being locally spherical.
\end{proof}

\section{CR manifolds with transverse symmetry}\label{S:transym}

Let $M=M^3$ be a three dimensional, strictly pseudoconvex CR manifold with  transverse symmetry. In other words, there is a smooth real vector field $T^0$ on $M$ such that
\begin{itemize}
\item[(a)]  $T^0$ is an {\it infinitesimal CR automorphism}, i.e., generates locally a 1-parameter family of CR automorphisms of $M$; and
\item[(b)]  $T^0$ is transverse to the complex tangent space $H_p:=\re T^{1,0}_pM$ at every point $p\in M$.
\end{itemize}
It is well known (see \cite{BER99a}) that (a) is equivalent to $T^0$ having the property that $[T^0,L_1]$ is a $(1,0)$-vector field for every $(1,0)$-vector field $L_1$. In particular, $T^0$ is the Reeb vector field for a uniquely determined contact form $\theta_0$, i.e., there is a contact form $\theta_0$ such that
\beq
\left<\theta_0,T^0\right> =1,\quad T^0\lrcorner d\theta_0=0.
\eeq
Indeed, it was proved in \cite{BRT85} that near a point $p\in M$, one can find local coordinates $x=(z,t)\in \bC\times \bR$, vanishing at $p$, and a local $(1,0)$-vector field $L^0_1$ spanning $T^{1,0}M$ near $p$ such that
\beq\label{transcoord}
T^0=\frac{\partial}{\partial t},\quad L^0_1=\frac{\partial}{\partial z}-f\frac{\partial}{\partial t},\quad \theta_0=dt+fdz+\bar fd\bar z,
\eeq
for some smooth function $f$ near $p=(0,0)$ such that $f(p)=0$ and $f$ is a function of $z$ alone, $f=f(z,\bar z)$. Thus, we have
\beq\label{Levi-1}
d\theta_0=(D\bar f-\bar Df)dz\wedge d\bar z.
\eeq
The strict pseudoconvexity implies that the purely imaginary function $D\bar f-\bar Df$ is nonzero. By replacing $T^0$ by $-T^0$ if necessary, we may assume that we have
\beq\label{udef}
D\bar f-\bar Df=ie^{2\phi},
\eeq
where $\phi=\phi(z,\bar z)$ is a smooth real-valued function. Thus, we may rewrite \eqref{Levi-1} as \beq\label{Levi-2}
d\theta_0=ie^{2\phi}\theta^1\wedge\theta^{\bar 1}, \quad \theta^1:=dz.
\eeq
The contact form $\theta_0$ defines a pseudohermitian structure \cite{Webster78} on $M$ and $(\theta_0,\theta^1, \theta^{\bar 1})$ is a local admissible coframe in this pseudohermitian structure; the reader is referred to \cite{Webster78} and \cite{Lee88} for basic facts regarding pseudohermitian structures. The fact that the Reeb vector field is an infinitesimal automorphism implies that the torsion $\tau_1=A_{11}\theta^1$ vanishes, and thus the connection form $\omega_1{}^1$ is identified via the structure equation for $d\theta^1$ and symmetry requirement following from \eqref{Levi-2}, respectively,
\beq\label{struc-1}
d\theta^1=\omega_1{}^1\wedge\theta^1, \quad \omega_{1\bar 1}+\omega_{\bar 1 1}=dh_{1\bar 1},
\eeq
where $\theta^1=dz$, $\omega_{\bar 1 1}=\overline{\omega_{1\bar 1}}$ and we use the Levi form $h_{1\bar 1}:=e^{2\phi}$ to raise and lower indices. Since $d\theta^1=d^2z=0$ and
\beq
dh_{1\bar 1}=d(e^2\phi)=e^{2\phi}(2D\phi\, dz+2 \bar D\phi\, d\bar z)
\eeq
we conclude then from \eqref{struc-1} that
\beq
\omega_{1\bar 1}:=h_{1\bar 1}\omega_1{}^1=2e^{2\phi} D\phi\, \theta^1,
\eeq
or equivalently
\beq\label{conn-1}
\omega_1{}^1 =2D\phi\,dz= 2D\phi\, \theta^1.
\eeq
The following proposition is then a direct consequence of the structure equation for $d\omega_1{}^1$:

\begin{proposition}\label{transKprop} The pseudohermitian scalar curvature $R:=R_1{}^1{}_1{}^1$ of $\theta_0$ is given by
\beq\label{curv-1}
R=-2 e^{-2\phi} D\bar D\phi.
\eeq
\end{proposition}

To be able to compare with the computations in Section \ref{S:general}, we renormalize $\theta:=e^{-2\phi}\theta_0$ so that \eqref{normtheta} holds (still with $\theta^1=dz$) with
\beq\label{btrans}
b=2D\phi, \quad b=b(z,\bar z).
\eeq
We observe at this point that we have an identity for $b=b(z,\bar z)$ of the same form as in Section \ref{S:hlb} with $2\phi$ in \eqref{btrans} playing the role of $-\log a$ in \eqref{b=Da/a}. Next, in order to compare with the computations in Section \ref{S:hlb}, we change the admissible coframe for the pseudohermitian structure of $\theta_0$ by $\hat \theta^1:=e^\phi dz=e^{\phi} \theta^1$. This normalizes the Levi form in this structure to $h_{1\bar 1}=1$. To compute the connection form $\hat\omega_1{}^1$ with respect to this coframe, we must consider the equations
\beq
d\hat \theta^1=d\phi\wedge \hat \theta 1=\hat\omega_1{}^1\wedge\hat \theta^1,\quad \hat\omega_{1\bar 1}+\hat \omega_{\bar 1 1}=0,
\eeq
which is easily seen to have the implication
\beq\label{hatconn}
\hat\omega_1{}^1=-(\partial\phi -\bar\partial\phi)=e^{-\phi}(\bar D\phi\, \hat \theta^{\bar 1} -D\phi\, \hat \theta^1).
\eeq
Next, we note that the dual $(1,0)$ vector field $\hat L_1$ corresponding to $ \hat \theta^1$ equals $e^{-\phi} L_1$. Thus, for any function $f$ that is independent of $t$, i.e., $f=f(z,\bar z)$, we have $\hat L_1f=e^{-\phi}Df$. We therefore observe that covariant differentiation of such $f$ with respect to the Tanaka--Webster connection in the coframe $\hat\theta^1$, in the the pseudohermitian structure of $\theta_0$ is the same as covariant differentiation of $f=f(z,\bar z)$ on the Riemann surface $X$ with coordinate $z$ and metric $ds^2=e^{2\phi}|dz|^2$ as in Section \ref{S:hlb};  E.g., if $f=f(z,\bar z)$ is a function on $M$ near $p=(0,0)$, then
\beq
f_{;1}=\hat L_1f=e^{-\phi}Df,\quad f_{;11}=e^{-2\phi}(D^2f-2(D\phi) Df),\ \ldots.
\eeq
With this observation, combined with Proposition \ref{transKprop} and the calculations yielding Theorem \ref{QisGauss}, we conclude that the following holds:

\begin{theorem}\label{QisGauss-trans} The invariant functions $Q$ and $Q_{;11}$ are related to the pseudohermitian scalar curvature $R$ of $M$ given by $\theta_0$ via:
\beq\label{GaussQs}
Q=-\frac{e^{4\phi}}{6}\ \frac{R_{;\bar 1\bar 1}}{\lambda\bar\lambda^3},\quad Q_{;11}= -\frac{e^{6\phi}}{6}\ \frac{R_{;\bar 1\bar 11 1}}{|\lambda\bar\lambda|^3}.
\eeq
\end{theorem}

\section{Proof of main result}\label{S:proof}

In this section, we shall prove the result stated in the introduction.

\noindent{\it Proof of Theorem $\ref{MainThmFeff}$.} It suffices to show that $b\eta_1=0$ on $M$ implies that $M$ is locally spherical, since the converse is clear, and moreover, if $D$ is simply connected and $M$ connected, it follows from the Riemann mapping theorem of Chern--Ji \cite{ChernJi96} that $D$ is biholomorphic to the unit ball $\bB^2$. Thus, to complete the proof of Theorem $\ref{MainThmFeff}$ it suffices to show that $M$ is locally spherical, provided $b\eta_1=0$ on $M$. Graham \cite{Graham87} showed that the space of CR invariants of weight 3 is one dimensional, spanned by $A^0_{44}$, and in particular $b\eta_1=4A^0_{44}$. Thus, if $b\eta_1=0$ on $M$, then it follows from Theorem \ref{Thm-wt3} that $Q_{;11}$ also vanishes on $M$. By Theorem \ref{QisGauss-trans}, we then conclude that $R_{;\bar 1\bar 1 1 1}=R_{;\bar 1\bar 1}{}^{\bar 1\bar 1}=0$ on $M$. We shall need the analog of Calabi's result used in the proof of Corollary \ref{CorKzzzz=0}:

\begin{proposition}\label{CalabiProp} If $f$ is a smooth function on a compact strictly pseudoconvex CR manifold $M=M^3$, and $f_{;\bar 1\bar 1}{}^{\bar 1\bar 1} =0$, then $f_{;\bar 1\bar 1}=0$.
\end{proposition}

\begin{proof} This is a simple integration by parts argument, using the divergence lemma (a.k.a. Stokes Theorem) in \cite{Lee88}:
\begin{multline}
\int_M|f_{\bar 1\bar 1}|^2 \theta\wedge d\theta=\int_M f_{\bar 1\bar 1} \bar f^{;\bar 1\bar 1} \theta\wedge d\theta=-\int_M f_{\bar 1\bar 1}{}^{\bar 1} \bar f^{;\bar 1} \theta\wedge d\theta =\int_M f_{\bar 1\bar 1}{}^{\bar 1\bar 1} \bar f \theta\wedge d\theta=0,
\end{multline}
which proves the proposition.
\end{proof}

Proposition \ref{CalabiProp} with $f=R$ now completes the proof of Theorem $\ref{MainThmFeff}$, in view of the first identity in \eqref{GaussQs} of Theorem \ref{QisGauss-trans}. \qed


\def\cprime{$'$}

\end{document}